\documentclass{amsart}

\usepackage[latin1]{inputenc}
\usepackage{amsfonts}
\usepackage{amsmath}
\usepackage{amsthm}
\usepackage{amssymb}
\usepackage{latexsym}
\usepackage{enumerate}

\newtheorem{theorem}{Theorem}[section]
\newtheorem{lemma}[theorem]{Lemma}
\newtheorem{proposition}[theorem]{Proposition}
\newtheorem{corollary}[theorem]{Corollary}

\newtheorem{question}[theorem]{Question}

\newtheorem{definition}[theorem]{Definition}

\numberwithin{equation}{section}

\begin{document}

\newcommand{\cc}{\mathfrak{c}}
\newcommand{\N}{\mathbb{N}}
\newcommand{\Q}{\mathbb{Q}}
\newcommand{\R}{\mathbb{R}}
\newcommand{\T}{\mathbb{T}}
\newcommand{\st}{*}
\newcommand{\PP}{\mathbb{P}}
\newcommand{\SSS}{\mathbb{S}}
\newcommand{\forces}{\Vdash}
\newcommand{\dom}{\text{dom}}
\newcommand{\osc}{\text{osc}}

\title[totally disconnected reflection of nonmetrizability]{On the problem of compact totally 
disconnected reflection of nonmetrizability}

\author{Piotr Koszmider}
\address{Institute of Mathematics, Polish Academy of Sciences,
ul. \'Sniadeckich 8,  00-656 Warszawa, Poland}
\email{\texttt{piotr.koszmider@impan.pl}}
\thanks{This research  was partially supported by   grant
PVE Ci\^encia sem Fronteiras - CNPq (406239/2013-4).} 

\begin{abstract} 
We construct a ZFC example of a nonmetrizable compact space $K$ such that every totally
disconnected closed subspace $L\subseteq K$ is metrizable. In fact, 
the construction can be arranged so that every nonmetrizable compact subspace
may be of fixed big dimension.
Then we focus on the problem if  a nonmetrizable compact space $K$ must
have a closed subspace with a nonmetrizable  totally
disconnected  continuous image. This question has several links with the
the structure of the Banach space $C(K)$, for example, by Holszty\'nski's theorem, if $K$ is
a counterexample, then $C(K)$ contains no isometric
copy of a nonseparable Banach space $C(L)$ for $L$ totally disconnected.
We show that  in the literature there are diverse consistent counterexamples, most 
eliminated by Martin's axiom and the negation of the continuum hypothesis, but
some consistent with it.  We analyze the above problem for a particular class of spaces.
OCA+MA however, implies the nonexistence of any
counterexample in this class but the existence of some other absolute example
remains open.
\end{abstract}

\maketitle

\section{introduction}

This paper is related to the question whether a  nonmetrizable space must have a nice, in some sense, 
nonmetrizable subspace. If the nice subspace that we seek means 
a subspace of small cardinality, positive consistent answers to this question  were obtained  by
Alan Dow and others, for example,   in \cite{twoapplications}, \cite{submodels}, \cite{dtw2}, 
\cite{semi}, \cite{vaughan}. 
When one restrict oneself to compact Hausdorff spaces the question if every nonmetrizable
compact Hausdorff space has a nonmetrizable subspace of cardinality $\omega_1$
has the positive answer in ZFC as proved by Alan Dow in \cite{empty}. Here we will 
ask about the reflection of the nonmetrizability for compact Hausdorff spaces
to another type of nonmetrizable subspaces or quotient spaces,
namely we want them to be totally disconnected and compact. Thus, the main questions
are:

\begin{question}\label{questions1} Suppose that $K$ is compact Hausdorff space which is nonmetrizable.
\begin{enumerate}
\item Is there $L\subseteq K$ which is compact, nonmetrizable and totally disconnected?
\item Is there a closed subspace $K'\subseteq K$ and
a continuous surjective map $\phi: K'\rightarrow L$  such that $L$ is nonmetrizable and
totally disconnected? 
\end{enumerate}
\end{question}
It is worthy to note that
in Question \ref{questions1} (2) (see \cite{opit}, Question 4 (1176)) instead of continuous images of 
closed subspaces we could consider  closed subspaces of continuous images (Lemma \ref{subquotient}).
Consistent examples providing positive answer to Question \ref{questions1} (1) have been
well known, for example,
assuming the continuum hypothesis (CH),
V. Fedorchuk showed that there are compact spaces where every infinite closed
subspace has big dimension (\cite{fedorchuk}), assuming $\diamondsuit$
M. E. Rudin and P. Zenor constructed a nonmetrizable manifold where
all closed subsets are metrizable or contain many copies of euclidean intervals
(\cite{zenor}, 3.14 of \cite{nyikos}).  It seems to be folkloric knowledge
that the Souslin continuum is another example. Assuming $\clubsuit$ it is possible to
construct a  $\T$-bundle over the long ray like in Example 6.17 of \cite{nyikos} whose one point
compactification  provides another example. Some of these examples are consistent with any cardinal arithmetic, but
some have continuous image, the compactification of the long ray, which contains 
a nonmetrizable totally disconnected subspace $[0,\omega_1]$. 
To obtain counterexamples to the second question from the above examples one needs 
to do a bit more work. We review these and other examples in the
context of Question \ref{questions1} (2) in Proposition \ref{properties}. 

In this note we focus especially on constructions of compact spaces  of certain concrete  type
which do not need to be locally compact as many of the above examples, which we call
 split compact spaces in the analogy to the usual split interval (see e.g. \cite{godefroy}).
Given a metrizable compact $M$, its points $\{r_\xi: \xi<\kappa\}$ for some 
cardinal $\kappa$ and the splitting continuous functions $f_\xi:M\setminus\{r_\xi\}\rightarrow 
K_\xi$ where $K_\xi$s are compact and metrizable we consider the split
$M$ induced by $(f_\xi)_{\xi<\kappa}$, for precise definition see \ref{splitdefinition}.
In particular for us a split interval has a more general meaning than the usual split interval,
to underline this difference we will talk about unordered split intervals in the nonclassical case.
Such topological constructions can be traced back to Fedorchuk's school 
and found many applications in topology and in particular dimension theory (see \cite{fedorchuksur}),
and were rediscovered by Koppelberg in the context of totally disconnected spaces.
Recently they  and similar spaces have been applied in functional analysis in 
the connected version in \cite{equilateral} and totally disconnected version 
in \cite{rolewicz}, \cite{biorthogonal}.

The  paper can be summarized as an attempt to construct spaces providing
negative answers to Questions \ref{questions1} (1) and (2) of the above form.
Our main results are:
{\it\begin{enumerate}[(a)]
\item  There is (in ZFC) a nonmetrizable compact Hausdorff space where every totally disconnected
compact subspace is metrizable. Our example  is  an unordered split interval. (Theorem \ref{zfcexample}).
\vskip 3pt
\item  Assuming the existence of a Luzin set\footnote{Recall that a Luzin set is an uncountable subset of the reals which meets every
nowhere dense set only on a countable subset.
Note that the assumption of the existence of a Luzin set is consistent with any cardinal arithmetics
(just add $\omega_1$ Cohen reals), it follows from CH and under the failure
of CH it contradicts Martin's axiom.}
there is an unordered split interval which is
a nonmetrizable compact Hausdorff space without a continuous image containing a 
nonmetrizable totally disconnected closed subspace.  (Theorem \ref{luzinexample}).
\vskip 3pt
\item The existence of  a compact space with no subspace with a  continuous image
which is a nonmetrizable and  totally disconnected 
 is consistent with Martin's axiom (MA) and the 
negation of CH. This is the Filippov split square, but 
split intervals or other examples can be arranged as well. (Theorem \ref{macon}).
\vskip 3pt
\item Assuming the Open Coloring Axiom\footnote{Recall that   OCA (see 8.0 of \cite{partition}) 
developed by Abraham, Shelah and Todorcevic says that for
any partition $[X]^2=K_0\cup K_1$ of a subset $X$ of the reals such that $K_0$ is open
in the product topology on $X\times X$  there is either an uncountable
$Y\subseteq X$ such that $[Y]^2\subseteq K_0$ or $X=\bigcup_{n\in \N}X_n$
where $[X_n]^2\subseteq K_1$ for each $n\in \N$. OCA is consistent with MA, 
implies the failure of CH and follows from the Proper Forcing Axiom PFA or Martin's
Maximum.} (OCA) 
every nonmetrizable  split compact space has a continuous image with a 
nonmetrizable totally disconnected closed subspace. (Theorem \ref{ocaresult}). 
\end{enumerate} }

Our approach of considering split compact spaces as in Definition \ref{splitdefinition}
to attack Question \ref{questions1} (1) turns out to be successful.
Our ZFC example from Theorem \ref{zfcexample}
which answers Question \ref{questions1} (1) is based on a simple combinatorial principle (Lemma \ref{cpol})  
discovered by K. Ciesielski and R. Pol in \cite{cpol} as Remark 7.1
(see also \cite{godefroypams}). It can be sometimes used to replace
an application  of $\clubsuit$ by a ZFC argument. Its more complicated versions
were used by A. Dow and coauthors in Examples 2.15 and 2.16 of \cite{djp} 
others were used in Section 5 of \cite{rn},
however, we feel that despite their simplicity these principles are not widely known.

However, our results concerning Question \ref{questions1} (2) show that attacking this question
with split compact spaces as in Definition \ref{splitdefinition} is not optimal
in the sense that taking this way it turns out that we end up facing a well-known and apparently harder problem
whether locally connected perfectly normal compact spaces must be metrizable (see e.g. \cite{kunenloc}).
Despite apparent much bigger flexibility of the split intervals
or in general split compact spaces compared to such locally compact examples,
already evident when considering the closure of the graph of $\sin(1/x)$, $x\in (0,1]$,
the combinatorial essence of interesting counterexamples turns out to be the same as for
locally compact examples (compare our Proposition \ref{weaklyluzin} and Theorem 2.5. of 
\cite{kunenloc}).  Also the OCA annihilates the examples the same way
adding a continuous image which has the one point compactification of
the uncountable discrete space as the totally disconnected subspace.
To avoid repetitions we decided to present the results  in the generality of 
split compact spaces (\ref{splitdefinition}) which   allows to 
 rely heavily in (c) and (d) on the results of K. Kunen from  \cite{kunenloc} concerning Fillipov's spaces.
As counterexamples to Question \ref{questions1} (2), to survive the impact of MA+$\neg$CH,
must be hereditarily Lindel\"of and hereditarily separable (see \ref{properties})
one perhaps could consider connected and ZFC versions of constructions like in \cite{rolewicz}
and \cite{biorthogonal} for which Kunen's OCA argument does not apply and the spaces
are still preimages of metric spaces with metrizable fibers.

The totally disconnected reflection of the nonmetrizability for compact spaces 
in the sense of  Question \ref{questions1}
has another  strong motivation coming from functional analysis.
The special role of Banach spaces of the form $C(K)$ for $K$ compact, Hausdorff
and totally disconnected in the general theory of Banach spaces has been evident
since the beginnings of this theory. Starting with Schreier's analysis of the space
$C([0,\omega^\omega])$ which answered Banach's question if nonisomorphic Banach spaces may have
isomorphic duals, through Johnson-Lindenstrauss', Haydon's, Talagrand's, Argyros' and other now
classical examples, the clear combinatorial structure of the Boolean algebra
$Clop(K)$ of the clopen subsets of $K$ and its generation of the dense subspace of $C(K)$
of simple functions served as a miraculous tool multiplying interesting examples and counterexamples
relevant in the general theory of Banach spaces.

Since the isomorphic classification of separable Banach spaces of the form $C(K)$
(equivalently for $K$ metrizable) due to Milutin, Bessaga and Pe\l czy\'nski (\cite{classification})
which implied that every such $C(K)$ is isomorphic to a $C(L)$ for $L$ totally
disconnected the issue whether this is the case for nonmetrizable  compact $K$s
has emerged (\cite{semadeni}). Despite some progress in this
direction, e.g. showing it for arbitrary compact
topological groups (\cite{diss}) it turned out only recently that there are
$C(K)$s not isomorphic to $C(L)$s or $L$ totally disconnected (\cite{few},
for further references see \cite{fewsur}) and that such $K$s can be relatively
nice like in \cite{rn}. Hence we cannot assume in the isomorphic theory
that all Banach spaces  $C(K)$ are given by totally disconnected compact $K$s.
So the next natural question is whether given a Banach space $C(K)$ we can associate
with it a $C(L)$ for $L$ totally disconnected and  compact such that $C(L)$ provides
some useful information about the $C(K)$. For example, one classical result of 
 S. Ditor is that there is $L$ 
of the same weight as $K$ such that $C(K)$ is $1$-complemented in $C(L)$ (\cite{ditor}).
In this context it is natural to ask the following:
\begin{question}\label{questions2}
Suppose that $K$ is a compact Hausdorff nonmetrizable space.
Is there compact Hausdorff nonmetrizable totally disconnected $L$ such that:
\begin{itemize}
\item $C(L)$ embeds linearly and isometrically as a Banach space into $C(K)$?
\end{itemize}
\end{question}
It remains open if the above question has consistently positive answer. On the other hand
compact spaces $K$ with no subspace with a continuous nonmetrizable 
totally disconnected image give the negative answer to Question \ref{questions2}. 
This follows from a result of Holszty\'nski
(\cite{holsztynski})  which says that isomorphic embedding of a $C(L)$ into $C(K)$ is 
always induced by continuous map of closed subset of $K$ onto $L$. Although 
in general the existence of such a map is not equivalent to the existence of the isometric embedding
note that if $K'\subseteq K$ and $\phi:K'\rightarrow L$ is a continuous surjection
and $L$ is nonmetrizable totally disconnected then there 
are (possibly nonlinear) isometries $\Psi: C(L)\rightarrow C(K)$ and
$\Psi^*: C(L)^*\rightarrow C(K)^*$ such that $\mu(f)=\Psi^*(\mu)(\Psi(f))$
for every $f\in C(L)$ and for every functional $\mu\in C(L)^*$.
They can be obtained by the Tietze theorem and  by the Hahn-Banach theorem.
More concretely, $\Psi(f)$ is a supremum norm preserving extension 
of $f\circ\phi$ from $K'$ to $K$ and $\Psi^*(\mu)$ is the Radon measure
on $K$ concentrated on $K'$ obtained by extending the functional 
on the subspace $\{f\circ \phi: f\in C(L)\}$ which corresponds to the 
functional on $C(L)$ defined by the measure $\mu$.
$\Psi$ and $\Psi^*$ can be quite useful for transferring the consequences
of known theorems proved for totally disconnected $L$s to general $K$s.
This is at least relevant for nice biorthogonal systems (see \cite{dzamonja}, \cite{cardinal}),
equilateral sets (\cite{equilateralcofk}, \cite{equilateral}) or sets separated by
more than one (\cite{kania}). Even in the negative direction we can make a new
observation concerning equilateral sets in Banach spaces (Corollary \ref{equilateral}).

Without mentioning we will often be using basic facts concerning compact spaces  like the equivalence
of the zero dimensionality and the total disconnectedness,  or the equivalence of the metrizability and
the existence of a countable family of continuous functions which separate the points
of the space or the dependence of
continuous functions in the products on countably many coordinates. 
We refer to the book \cite{engelking} of R. Engelking  for these issues. All
topological spaces considered in this paper are Hausdorff.

\section{Splitting compact metrizable spaces}

Given a cardinal $\kappa$, a compact metrizable $M$
with no isolated points , a sequence $(K_\xi)_{\xi<\kappa}$
of metrizable compact spaces  and  functions $f_\xi: [0,1]\setminus\{r_\xi\}\rightarrow K_\xi$
 for some distinct  $r_\xi\in M$
for $\xi<\kappa$   we may define a natural  version of the 
split interval (see e.g. \cite{godefroy}) which can be naturally embedded in the product 
space $L\times\Pi_{\xi<\kappa}K_\xi$.

\begin{definition}\label{splitdefinition}{\rm\cite{equilateral}} Let $\kappa\leq 2^\omega$  be a cardinal.
Let $M$ compact Hausdorff metrizable and with no isolated points. Let $K_\xi$ for $\xi<\kappa$ be compact Hausdorff metrizable spaces.
Suppose that $\{r_\xi: \xi<\kappa\}$ consists of distinct elements of  $M$,
$f_\xi: M\setminus\{r_\xi\}\rightarrow K_\xi$ is
 a continuous function such that $f_\xi[U\setminus\{r_\xi\}]$
is dense in $K_\xi$  for every open neighbourhood $U$ of $r_\xi$ for every  $\xi<\kappa$.  

 A  split $M$
 induced by $(f_\xi)_{\xi<\kappa}$ 
is the subspace $K$ of $M^{\{\st\}}\times \Pi_{\xi<\kappa}K_\xi$ consisting of points of the form 
\[\{x_{\xi, t}, : \xi<\kappa, t\in K_\xi\}\cup\{x_r: r\in M\setminus\{r_\xi: \xi<\kappa\}\}, \]
where
\begin{enumerate}
\item $x_{\xi, t}(\st)=r_\xi$, $x_{\xi, t}(\xi)=t$ and $x_{\xi, t}(\eta)=f_\eta(r_\xi)$
 if $\eta\in \kappa\setminus\{\xi\}$,
\item $x_r(\st)=r$ and $x_r(\xi)=f_\xi(r)$ for all 
$r\in M\setminus\{r_\xi: \xi<\kappa\}$ and $\xi<\kappa$.
\end{enumerate}
Under these assumptions we will use the following notation and terminology:
\begin{enumerate}
\item $U_K=\{x\in K: x(\st)\in U\}$ for $U\subseteq M$,
\item $U_{K,\xi}=\{x\in K: x(\xi)\in U\}$  for any $U\subseteq K_\xi$ and every $\xi<\kappa$,
\item $R_\xi=\{ x_{\xi,t}: t\in K_\xi\}$ for all $\xi<\kappa$,
\item $f_\xi$s will be called the splitting functions.
\end{enumerate}
\end{definition}

Thus, the classical split interval  $S$ is obtained by choosing 
$M=[0,1]=\{r_\xi: \xi<2^\omega\}$, $K_\xi=\{0,1\}$,
 $f_\xi: [0,1]\setminus\{r_\xi\}\rightarrow \{0,1\}$ defined by
$f_\xi(r)=0$ if $r<r_\xi$ and $f_\xi(r)=1$ if $r>r_\xi$.
In \cite{equilateral} we considered $f_\xi: [0,1]\setminus\{r_\xi\}\rightarrow [0,1]$ 
modeled after $f_\xi(x)=\sin({1\over{|x-r_\xi|}})$.
In \cite{filippov} (cf. \cite{kunenloc}) V. Filippov considered $M=[0,1]^2$, 
$f_\xi: [0,1]^2\setminus\{r_\xi\}\rightarrow \T$   given by $f_\xi(x)={{x-r_\xi}\over {||x-r_\xi||}}$
where, $\T$ is the unit sphere in $\R^2$ and  $\{r_\xi: \xi<\kappa\}=E$ is a chosen 
subset of $[0,1]^2$. In \cite{kunenloc} Kunen calls
this space the Filippov space and denotes it $\Phi_E$, we will
follow this convention.

Note that it follows from the definition of the split  $M$ that
the only point $x$ of $K$ such that $x(\st)=r$ is $x_r$ if 
$r\in M\setminus\{r_\xi: \xi<\kappa\}$ and that the only points
 $x$ of $K$ such that $x(\st)=r_\xi$ for $\xi<\kappa$ are the points of $R_\xi$ that is  $x_{\xi, t}$s 
for $t\in K_\xi$ and these points differ just at the $\xi$-th coordinate and are equal 
on all other coordinates of the product. It is clear that $R_\xi$ is always
a homeomorphic copy of $K_\xi$.  

\begin{proposition}\label{splitproperties} Let $\kappa$,  $M$, $K_\xi$s  and
 $\{r_\xi: \xi<\kappa\}$ be as in Definition \ref{splitdefinition}. Suppose
that $M$ and $K_\xi$s for $\xi<\kappa$ are moreover connected.
Let
$K\subseteq [0,1]^{\{\st\}}\times\Pi_{\xi<\kappa} K_\xi$ is a  split $M$ induced by 
$(f_\xi)_{\xi<\kappa}$. Then
\begin{enumerate} 
\item $K$ is a compact Hausdorff space,
\item $K$ is connected,
\item  $K$ is first countable and 
\[\{U^n_K: n\in \N\}\]
 forms a basis at $x_r$ for each $r\in M\setminus\{r_\xi:\xi<\kappa\}$,
where $(U^n: n\in \N)$ is a basis at $r$ in $M$.
and 
\[\{U^n_K\cap V^n_{K,\xi} : n\in \N\}\]
 forms a basis at $x_{\xi, t}$ 
for each $t\in [-1,1]$ and each $\xi<\kappa$,  where $(V^n: n\in \N)$ is a basis at $t$ in $K_\xi$.
\item $K^2$ has a discrete set of cardinality $\kappa$, if all $K_\xi$s have at least two points,
\item $C(K)$ has a biorthogonal system of cardinality $\kappa$.
\end{enumerate}
\end{proposition}
\begin{proof}
Like in Proposition 2.3 of \cite{equilateral}.
\end{proof}

\section{ZFC examples}

The following lemma is  due to K. Ciesielski and R. Pol (Remark 7.1 of \cite{cpol})
we provide its proof for the convenience of the reader.

\begin{lemma}\label{cpol} There is a collection $\{(s^\xi_n)_{n\in \N}:  \xi<2^\omega\}$ of
sequences of  the reals from $[0,1]$
and a collection $\{r_\xi: \xi<2^\omega\}$ of (distinct)  reals from $[0,1]$
such that:
\begin{enumerate}
\item for each $\xi<2^\omega$ the sequence $s_n^\xi$ converges to $r_\xi$,
\item  for every uncountable $X\subseteq [0,1]$ there is $\xi<2^\omega$ such that
$\{s_n^\xi: n\in \N\}\subseteq X$.
\end{enumerate}
\end{lemma}
\begin{proof}
Enumerate all countable subsets of $[0,1]$ with uncountable closures as $(A_\xi:\xi<2^\omega)$.
Construct $s_n^\xi$ and  $r_\xi$ by recursion on $\xi<2^\omega$. 
Suppose that we are done till $\xi<2^\omega$. As the closure of $A_\xi$ is uncountable,
as a closed subset of $[0,1]$ it must contain a copy of a Cantor set, i.e., its closure  has
cardinality $2^\omega$. So choose $r_\xi$ in the closure of $A_\xi$ distinct than
all $r_\eta$ for $\eta<\xi$, then choose $s_n^\xi\in A_\xi$ which converges to $r_\xi$.
Given any uncountable set $X\subseteq [0,1]$, its closure is uncountable, and there
is a dense countable subset $A$ of $X$. It follows that $A=A_\xi$ for some $\xi<2^\omega$
and so $\{s_n^\xi: n\in \N\}\subseteq X$.
\end{proof}

\begin{theorem}\label{zfcexample} There is a compact nonmetrizable space where all
totally disconnected subspaces are metrizable. There are such spaces which have
subspaces with continuous nonmetrizable totally disconnected images.
\end{theorem}
\begin{proof} Fix $\{(s^\xi_n)_{n\in \N}:  \xi<2^\omega\}$ and 
 $\{r_\xi: \xi<2^\omega\}$ as in Lemma \ref{cpol}. Define a split interval $K$ by
defining a splitting function
$f_\xi: [0,1]\setminus \{r_\xi\}\rightarrow [-1,1]$ in such a way that
each rational number in $[-1, 1]$ is assumed on the set $\{s_n^\xi: n\in \N\}$
infinitely many times.  Let $K\subseteq [0,1]^{\{*\}}\times[-1,1]^{2^\omega}$ be 
the (unordered) split interval induced by $(f_\xi)_{\xi<2^\omega}$.
We will show that no nonmetrizable closed subspace of $K$ is  totally disconnected.

Let $L\subseteq K$ be  nonmetrizable and compact. 
First note that $X=\{r\in [0,1]: r=x(*), x\in L\}$ must be uncountable.
Indeed, otherwise  there is a countable $A\subseteq 2^\omega$ such that
$L$ is a subset of
$$Y=\{x_{\xi,t}: \xi\in A, t\in [-1,1]\}\cup \{x_r: r\in X\setminus \{r_\xi: \xi<2^\omega\}\}.$$
Since the coordinates from $\{*\}\cup A$ separate the points of $Y$
by Proposition  \ref{splitproperties} (3), they separate the points of $L$, and so
$L$ is metrizable, a contradiction.

Now, by  Lemma \ref{cpol} there is $\xi<2^\omega$ 
such that $\{s_\xi^n: n\in \N\}\subseteq X$ and $s_n^\xi$ converges to $r_\xi$.
If $s_n^\xi=r_\eta$ for some $\eta<2^\omega$, then there is $t\in [-1,1]$
such that $x_{\eta, t}\in L$, in this case put $y_n=x_{\eta, t}$. 
If $s_n^\xi\in [0,1]\setminus\{r_\eta: \eta<2^\omega\}$, then put
$y_n=x_{s_n^\xi}$. In any case we have $y_n\in L$ and $y_n(*)=s_n^\xi$
for every $n\in \N$. 

We will show that $R_\xi\subseteq \overline{\{y_n: n\in \N\}}$ which will complete the proof
of the theorem as $\{y_n: n\in \N\}\subseteq L$ and $R_\xi$ is a  homeomorphic copy of 
$[-1,1]$ and so connected.  Take $x_{\xi, u}\in R_\xi$ for some
$u\in (a, b)\subseteq [-1,1]$ for $-1\leq a<b\leq 1$ and consider $(a, b)_{K,\xi}\cap
(r_\xi-1/k, r_\xi+1/k)_K$  for some $k\in \N$. By the construction of the splitting functions there is $n\in \N$
such that $s_n^\xi\in (r_\xi-1/k, r_\xi+1/k)$ and $f_\xi(s_n^\xi)\in (a,b)$, hence
$y_n\in (a, b)_{K,\xi}$ and $y_n\in (r_\xi-1/k, r_\xi+1/k)_K$. So
$x_{\xi, u}\in \overline{\{y_n: n\in \N\}}$ by Proposition \ref{splitproperties} (3). 
Since $x_{\xi, u}\in R_\xi$ was arbitrary
we obtain that $R_\xi\subseteq L$, and so $L$ is not totally disconnected.

To obtain a version which has a subspace with a nonmetrizable totally disconnected
continuous image note that in the Lemma \ref{cpol} we may pick all $\{r_\xi: \xi<2^\omega\}$
from a fixed Bernstein set $B\subseteq \R$ , that  is a set  such that both
$B$ and $\R\setminus B$ intersect  every closed uncountable
subset of the reals on a set of cardinality continuum. Fix an 
uncountable nowhere dense closed $F\subseteq [0,1]$ with no isolated points.
If we choose $r_\xi$s as in the above construction only from $B$,
for $r\in F\setminus B$ we are free to choose $f_\xi: [0,1]\setminus\{r\}\rightarrow [-1,1]$.
So do it in such a way that $f_\xi[F]=\{0\}$. Consider
$$K'=\{x\in K: x(*)\in F,   (r_\xi\in F\setminus B\ \ \Rightarrow\ x(\xi)\in \{0,1\})\}$$ 
and $\phi: K\rightarrow [-1,1]^{F\setminus B}$ given by $\phi(x)(t)=x(\xi)$
for $x\in K$ and $t=r_\xi\in F\setminus B$.  We note that $K'$ is closed, $\phi[K']$ is
totally disconnected and $\phi(x_{\xi,i})(t)=0$ if $t\not=\xi$, $i=0,1$ and $\phi(x_{\xi,i})(t)=1$
if $t=\xi$, $i=1$, so $\phi[K']$  cannot have a countable family of continuous functions
which separate the points, and so is nonmetrizable.

\end{proof}

Slightly modifying the above construction we can obtain:

\begin{theorem} Let $n\in \N\cup\{\infty\}$. There is a nonmetrizable compact space where all
nonmetrizable compact subspaces are of dimension $n$.
\end{theorem}
\begin{proof}
Fixing $n\in \N\cup\{\infty\}$ consider a split interval $K$ 
induced by splitting functions $f_\xi: [0,1]\setminus \{r_\xi\}\rightarrow [0,1]^n$ where by $[0,1]^\infty$ 
we mean $[0,1]^\N$.  Use Lemma \ref{cpol} as in Theorem
\ref{zfcexample} to define  $f_\xi$ so that 
each point from a fixed countable dense subset of $[0,1]^n$ is assumed on the set $\{s_n^\xi: n\in \N\}$
infinitely many times. The same argument as in the proof of Theorem
\ref{zfcexample} shows that any nonmetrizable subspace of $K$ includes some
$R_\xi$ which is homeomorphic to $[0,1]^n$.
\end{proof}

\section{Totally disconnected nonreflection in all continuous images}

We will be often using the following lemma without mentioning it:

\begin{lemma}\label{subquotient}
Let $K$ be a compact space and $\mathcal K$ the smallest class of compact
spaces containing $K$ which is closed under taking subspaces and under taking continuous images.
Then $\mathcal K$ is equal to the class of all subspaces of all continuous images of $K$
and it is equal to the class of all continuous images of all subspaces of $K$.
\end{lemma}
\begin{proof}
It is enough to show that both of the latter classes are equal since the first one
is closed under taking subspaces and the second one is closed under taking
continuous images.
If $L$ is a subspace of $K'$ and $\phi: K\rightarrow K'$ is a continuous surjection,
then $\phi^{-1}[L]\subseteq K$ is a subspace of $K$ which maps onto $L$.
If $\phi:K'\rightarrow L$ is a continuous surjection and $K'$ is a subspace
of $K$ consider $L$ as a subspace of $[0,1]^\kappa$ for some cardinal $\kappa$.
By applying the Tietze theorem to the compositions $\pi_\alpha\circ\phi$, where
$\pi_\alpha$ is the projection from $[0,1]^\kappa$ onto its $\alpha$-th coordinate for
$\alpha<\kappa$ we obtain an extension $\psi: K\rightarrow \psi[K]\subseteq[0,1]^\kappa$
such that $L$ is a subspace of its image $\psi[K]$.
\end{proof}

\begin{proposition}\label{properties}
If $K$ is a compact nonmetrizable  space such that all totally disconnected
subspaces of all continuous images of $K$ are metrizable, then 
\begin{enumerate}
\item $K$ has no uncountable discrete subspace,
\item $K$ is not Eberlein compact,
\item $K$ is not Rosenthal compact.
\end{enumerate}
However, it is consistent that there are such spaces
which are:
\begin{enumerate}
\setcounter{enumi}{3}
\item  not hereditarily Lindel\"of, or
\item not hereditarily separable,  or
\item Corson compact.
\end{enumerate}
Assuming Martin's axiom and the negation of the continuum hypothesis
whenever $K$ is as above, then
\begin{enumerate}
\setcounter{enumi}{6}
\item $K$ does not carry a measure of uncountable type,
\item $K$ is hereditarily separable,
\item $K$ is hereditarily Lindel\"of,
\item $K$ is not Corson compact.
\end{enumerate}
\end{proposition}
\begin{proof}
(1) Suppose that $K$ has 
an uncountable
discrete subspace $\{x_\alpha: \alpha<\omega_1\}$. 
Denote by $F_\alpha=\{x_\alpha\}$ and $G_\alpha=\overline{\{x_\beta: \beta\not=\alpha\}}$
 and consider continuous
functions $f_\alpha: K\rightarrow [0,1]$ such that $f_\alpha|F_\alpha=0$ and $f_\alpha|G_\alpha=1$
for all $\alpha<\omega_1$.
Define $\phi: K\rightarrow [0,1]^{\omega_1}$ by $\phi(x)(\alpha)=f_\alpha(x)$ for
every $\alpha<\omega_1$ and every $x\in K$.
Note that $\{0,1\}^{\omega_1}\cap \phi[K]$ is nonmetrizble, totally disconnected and compact.

(2) To conclude that $K$ cannot be an Eberlein compact recall that 
nonmetrizable Eberlein compact spaces are not c.c.c (Corollary 4.6. \cite{rosenthalacta}) and
use (1).

(3) To conclude that $K$ cannot be a Rosenthal  compact recall that nonmetrizable Rosenthal compacta
contain either a copy of the split interval (the classical one) which is totally
disconnected or an uncountable discrete
subset (Theorem 4 of \cite{stevobaire}).

(4) To see that there  are consistently $K$s as above which are not hereditarily Lindel\"of
one may consider the one
point compactification $K$ of a version of
the nonmetrizable manifold obtained by M. E. Rudin and P. Zenor in \cite{zenor} from $\diamondsuit$.
To take care of continuous images of the subspaces, one needs to 
modify however, the construction so that, for example, the closed cometrizable subspaces 
look like the entire space. Let us sketch such a simplified 
construction
of   a connected version of an Ostaszewski space from $\diamondsuit$ (\cite{ostaszewski})
which works for our purpose.
It can be described in the language similar to our unordered split interval:
define an inverse limit system $K_\alpha\subseteq [0,1]^\alpha$ with
$\alpha\leq\omega_1$ containing the point $0^\alpha$ as a nonisolated point. Given $K_\alpha$
define $K_{\alpha+1}\subseteq K_\alpha\times[0,1]^\alpha$ as
the union of $\{0^\alpha\}\times[0,1]$ and the graph of a continuous $f_\alpha: K_\alpha
\setminus\{0^\alpha\}\rightarrow [0,1]$ such that $f_\alpha[U\setminus\{0^\alpha\}]=[0,1]$
for any neighbourhood of $0^\alpha$. The obtained $K=K_{\omega_1}$
contains points of the form $x_r$ for $r\in (0,1]$ and ${0^\alpha}^\frown x_r$ for $r\in (0,1]$  and
point $0^{\omega_1}$.  We have $x_r(\alpha)=f_\alpha(r)$ for all
$r\in (0,1]$ and $({0^\alpha}^\frown x_r)(\beta)=f_\beta(r)$ for $\alpha<\beta<\omega_1$.
At stage $\alpha$ if the $\alpha$-th term of the $\diamondsuit$-sequence  codes a 
subset of $K_\alpha$ which has $0^\alpha$ in the closure we make sure that
the $f_\alpha$ assumes a dense set of values in $[0,1]$ on the intersection of the subset with any
neighbourhood of $0^\alpha$. This way $\{{0^\alpha}^\frown x_r: r\in[0,1]\}$ is in the closure
of the set coded by the $\alpha$-th term of the $\diamondsuit$-sequence. 
Besides this we also require that $f_\alpha$
 assumes a dense set of values on the intersection of sets coded by the previous
$\beta$-th terms of the $\diamondsuit$-sequence for $\beta<\alpha$ with any
neighbourhood of $0^\alpha$. This can be arranged using the recursive argument.
As in the case of the usual Ostaszewski construction we  conclude that 
for every nonmetrizable closed subset of $K$ there is $\alpha<\omega_1$
such that $K$ contains all  points $x$ of
$K$ satisfying $x|\alpha=0^\alpha$. This set is connected. So it must be
collapsed to a point by any continuous surjection onto
a totally disconnected compact space. It is not difficult
to see that the rest of the space may give at most metrizable image.
Since $K$ is compact with a point of uncountable character, $K$
is not hereditarily Lindel\"of. $K$ is actually an $S$-space.

(5) To obtain an $L$-space having the properties of $K$
consider a Souslin line. The separable subspaces of $K$ are metrizable
and $K$ is an $L$-space (\cite{merudin}). What follows is based on a standard argument
going back to Kelley (\cite{kelley}). 
Let $K'\subseteq K$ and $\phi: K'\rightarrow L$ with $L$ totally disconnected.
Consider the family $\mathcal I$ of maximal open intervals $I$  in $K$
such that $I\cap K'=\emptyset$ and the family $\mathcal J$ of maximal open intervals $J$  in $K$
such that $J\subseteq K'$. Since intervals are connected, for any interval
$J\in \mathcal J$ the set  $\phi[J]$ has one element
equal to $\phi(x)$ where  $x$ is any of the endpoints of $J$
as $L$ is totally disconnected. It  follows that $L=\phi[K']=\phi[K'\setminus \bigcup \mathcal J]$.
$K'\setminus \bigcup \mathcal J$ is nowhere dense and the endpoints of the 
intervals from $\mathcal I\cup \mathcal J$ form a dense subset of $K'\setminus \bigcup \mathcal J$.
As $K$ is c.c.c. $\mathcal I\cup \mathcal J$ is countable and so $K'\setminus \bigcup \mathcal J$
is separable and so metrizable and hence $L$ is metrizable as well.

(6) It follows from a result of Shapirowski (Corollary 10' of \cite{shapirovskii}) that any
compact space of countable tightness, in particular, the Souslin line as in (5)
can be continuously irreducibly mapped onto a  Corson compact space. 
Such an irreducible image cannot be metrizable, because it would be separable, and so
the closure of preimage of the dense countable set would contradict the irreducibility.
By Lemma \ref{subquotient}
this Corson compact must have the property that all continuous images of its subspaces
which are totally disconnected are metrizable.

(7) Result of Fremlin \cite{fremlin} says that under Martin's axiom and the negation
of the continuum hypothesis a compact space which carries a Radon measure
of uncountable type maps continuously onto $[0,1]^{\omega_1}$ which contains
a nonmetrizable compact totally disconnected $\{0,1\}^{\omega_1}$.

(8) and  (9). Under Martin's axiom and the negation of the continuum hypothesis
being hereditarily Lindel\"of and being hereditarily separable are equivalent for
compact spaces (\cite{szentmiklossy}). So suppose that a compact $K$ is not hereditarily Lindel\"of
and so has a right separated uncountable sequence. Such a  sequence is
locally countable and so by \cite{balogh} if $K$ is countably tight,  this
sequence is a countable union of discrete subspaces. In any case $K$ has
an uncountable discrete space, so (1) can be applied.

(10) Under Martin's axiom and the negation of the continuum hypothesis nonmetrizable
Corson compacta have  uncountable discrete subspaces (Corollary 5.6. of \cite{marciszewski}).
\end{proof}


Below we present our paradigmatic example of a compact nonmetrizable $K$ 
where all totally disconnected continuous images of closed subspaces are metrizable:

\begin{theorem}\label{luzinexample} Suppose that $\{r_\xi: \xi<\kappa\}\subseteq [0,1]$
is an enumeration of a Luzin set and $f_\xi: [0,1]\setminus\{r_\xi\}\rightarrow [-1,1]$
be any splitting functions. Let  $K\subseteq [0,1]^{\{*\}}\times[-1,1]^{2^\omega}$ be 
the unordered split interval induced by $(f_\xi)_{\xi<2^\omega}$.
Then no nonmetrizable compact subspace of a continuous image of $K$ is totally
disconnected.
\end{theorem}
\begin{proof} Let $L\subseteq K$ be a closed subspace and 
let $\phi: L\rightarrow L'$ be a continuous surjection. We will  show that 
if $L'$ is nonmetrizable, then it is not totally disconnected.

First note that $A=\{\xi<\kappa: |\phi[R_\xi\cap L]|>1\}$ must be uncountable if 
$L'$ is to be nonmetrizable. Indeed if $A$ were countable, consider
$\psi: L\rightarrow [0,1]^{\{*\}}\times[-1,1]^{A}$ defined by
$\psi(x)=x|(\{*\}\cup A)$ and note that $\phi$ is constant
on sets of the form $\psi^{-1}(\{y\})$  for $y\in [0,1]^{\{*\}}\times[-1,1]^{A}$
because they  are $L\cap R_\xi$ for $\xi\in 2^\omega\setminus A$ or
singletons. It follows that there is a  $\theta: [0,1]^{\{*\}}\times[-1,1]^{A}
\rightarrow L'$ such that $\phi=\theta\circ \psi$. Since $\psi$ is
a closed onto mapping (2.4.8. of \cite{engelking}) it is a quotient map
and so $\theta$ is continuous (2.4.2. of \cite{engelking}). 
But the codomain
of $\psi$ is metrizable, and so $L'$ would be metrizable as well. This
proves that $A$ cannot be countable.

By the defining property of the Luzin set there is an interval
$(a,b)$ for $0<a<b<1$ such that $\{r_\xi: r_\xi\in (a,b), \xi\in A\}$ 
is dense in $(a,b)$. We will show that for every $\xi \in A$ such that 
$r_\xi\in (a, b)$ we have $R_\xi\subseteq L$. This will be enough to
conclude the theorem since $\phi[R_\xi\cap L]$ has at least two points by
the definition of $A$ and $R_\xi$ is a copy of $[-1,1]$ and so connected,
and hence $\phi[R_\xi\cap L]=\phi[R_\xi]$ is a connected subspace of $L'$
which is not degenerated to one point, hence $L'$ is not totally disconnected.

Take $x_{\xi, u}\in R_\xi$ for some $\xi\in A$ such that $r_\xi\in (a, b)$ and 
$u\in (c, d)\subseteq [-1,1]$ for $-1\leq c<d\leq 1$ and consider $(c, d)_{K,\xi}\cap
(r_\xi-1/k, r_\xi+1/k)_K$  for some $k\in \N$ such that
$(r_\xi-1/k, r_\xi+1/k)\subseteq (a,b)$.
By the density of $\{r_\eta: \eta\in A, |R_\eta\cap L|>1\}$ in $(a,b)$ and the 
property of  the splitting functions  in Definition \ref{splitdefinition} we can find $\eta\in A$  such that
such that $r_\eta\in (r_\xi-1/k, r_\xi+1/k)$ and $f_\xi(r_\eta)\in (c,d)$.
Take $t\in [-1,1]$ such that $x_{\eta,t}\in L$.
We have that 
$x_{\eta,t} \in (c, d)_{K,\xi}$ and $x_{\eta, t}\in (r_\xi-1/k, r_\xi+1/k)_K$. So
$x_{\xi, u}\in L$. Since $x_{\xi, u}\in R_\xi$ was arbitrary
we obtain that $R_\xi\subseteq L$, and conclude as above  that $L'$ is not totally disconnected.

\end{proof}

\begin{proposition}\label{weaklyluzin} Let $\kappa$, $M$, $K_\xi$s, 
$\{r_\xi: \xi<\kappa\}$
and $f_\xi: M\setminus\{r_\xi\}\rightarrow K_\xi$ be as in Definition \ref{splitdefinition}.
Suppose moreover that all $K_\xi$s are connected.
Let  $K\subseteq M^{\{*\}}\times \Pi_{\xi<\kappa} K_\xi$ be 
the  split $M$ induced by $(f_\xi)_{\xi<\kappa}$.
Then the following are equivalent:
\begin{enumerate}
\item No nonmetrizable compact subspace of any continuous image of $K$ is totally
disconnected,
\item For every uncountable $A\subseteq \kappa$
the set $A_\xi=\{f_\xi(r_\eta):  \eta\in A\}$ is dense in $K_\xi$
for all but countably many $\xi\in A$,
\item $K$ is hereditarily Lindel\"of,
\item $K$ is hereditarily separable,
\item $K$ has no uncountable discrete subspace.
\end{enumerate}
\end{proposition}
\begin{proof}
All the above conditions imply (2): Suppose that there is an uncountable $A\subseteq \kappa$
and an open set $V_\xi\subseteq K_\xi$ such that
 $A_\xi\cap V_\xi=\emptyset$ for all $\xi\in A$. Choose $t_\xi\in V_\xi$ for each $\xi\in A$.
Then $\{x_{\xi, t_\xi}: \xi\in A\}$ is discrete as witnessed by
the neighbourhoods $V_{K, \xi}$ of $x_{\xi, t_\xi}$ as 
$f_\xi(r_\eta)\not\in V_\xi$ for $\eta\not=\xi$. Hence by Proposition \ref{properties} (1)
$K$ has a continuous image with a compact totally disconnected nonmetrizable subspace.

(2) implies the following (2'): For every uncountable $A\subseteq \kappa$
the set $A_{\xi, U}=\{f_\xi(r_\eta):  \eta\in A, r_\eta\in U\}$ is dense in $K_\xi$
for all but countably many $\xi\in A$ and any open $U\subseteq M$ containing $r_\xi$.
Otherwise, using the fact that $M$ is second countable, we would 
obtain an uncountable $A'\subseteq A$ and a fixed $U$ containing $r_\xi$s 
for $\xi\in A'$ such that $A_\xi=\{f_\xi(r_\eta):  \eta\in A, r_\eta\in U\}\supseteq
\{f_\xi(r_\eta):  \eta\in A'\}$ is not dense in $K_\xi$ for any $\xi\in A'$ contradicting (2).

(2') implies (1).
Suppose that $\phi: L\rightarrow L'$ is surjective and $L\subseteq K$ with $L'$
nonmetrizable. As in the proof of Theorem \ref{luzinexample} we obtain
an uncountable $A\subseteq \kappa$ such that $|\phi[R_\xi\cap L]|>1$ 
for all $\xi\in A$.  Using (2') find $\xi_0\in A$ such that $A_{\xi_0, U}$ is dense in $K_{\xi_0}$
for every open $U\subseteq M$ containing $r_{\xi_0}$.
It follows from the definition
of splitting functions (\ref{splitdefinition}) that 
$R_{\xi_0}\subseteq \overline{\{x_{\eta, t}\in L: \eta\in A\setminus \{\xi_0\} \}}\subseteq L$. But
$R_{\xi_0}$ is connected as a homeomorph of $K_\xi$ and so its continuous image which has more
than two points witnesses the fact that $L'$ is not totally disconnected.

(2') implies (3), (4) and so (5).  First let us prove that $K$ is hereditarily separable.
Assume $X\subseteq K$. First assume that $x(\st)=r_\xi$ for no  $x\in X$ nor $\xi<\kappa$.
Then the function sending  $x(\st)$  to $x\in X$ is continuous, as $M$ is second countable
it is hereditary separable so $X$ must be separable as well.
Now assume that $x(\st)=r_\xi$ for some $\xi<\kappa$ for all $x\in X$.
Let $A$ be the set of all $\xi\in \kappa$ such that there is
$x\in X$ satisfying $x(\st)=r_\xi$. By (2') we may assume that
$A_{\xi, U}$ is dense in $K_\xi$ for all $\xi\in A$ and open $U\subseteq M$
such that $r_\xi\in U$
(what is removed is
a countable union of subsets of copies of $K_\xi$ and so hereditarily separable).
Take a countable dense $D\subseteq \{r_\xi: \xi\in A\}$. It follows that
$R_\xi\subseteq {\overline{\{x_{\xi, t_\xi}: r_\xi\in D\}}}$ for
any choice of $t_\xi\in K_\xi$, in particular for such a choice that $x_{\xi, t_\xi}\in X$.
This gives a  countable dense subset of $X$. Combining the cases we obtain
a countable dense subset of $X$  in the general case.

To prove that $K$ is hereditarily Lindel\"of, assume that $X\subseteq K$. As
before we may assume that $x(\st)=r_\xi$ for some $\xi<\kappa$ for all $x\in X$.
Using the fact that $M$ is second-countable we may consider only open
covers $\mathcal U$ of $X$ consisting of sets (see \ref{splitproperties}) of
the form $U_{K,\xi}^\xi$ for $U^\xi\subseteq K_\xi$. These sets are 
unions of $V_K$ for an open in $M$ set $V=f_\xi^{-1}[U^\xi]$ and $\{x_{\xi, t}: t\in  U^\xi\}$. 
The collection of such $V$s has countable subcover as $M$ is hereditarily Lindel\"of.
So it remains  to cover the union of the sets $\{x_{\xi, t}: t\in  U^\xi\}$ such that
$r_\xi\not\in f_\eta^{-1}[U^\eta]$ for any $\eta\not=\xi$ and $U_{K,\eta}^\eta\in \mathcal U$.
If the set $A$ of such $\xi$s were uncountable we would have 
$\{f_\eta(r_\xi): \xi\in A\}\cap U^\eta=\emptyset$ for $\eta\in A$ which would contradict (2).
But if $A$ is countable we easily find a countable subcover using the hereditarily
Lindel\"of property of $K_\xi$s.
\end{proof}

\begin{theorem}\label{macon} It is consistent with MA$+\neg$ CH that 
there is  a nonmetrizable compact space with no
nonmetrizable totally disconnected subspace in any of its continuous images.
\end{theorem}
\begin{proof} We will use  the result of Kunen from \cite{kunenloc} (Theorems 2.5 and 3.3.)
which says that it is consistent with MA$+\neg$ CH  that there is 
an $E=\{r_\xi:\xi<\omega_1\}\subseteq [0,1]^2$ and a Filippov space $\Phi_E$ which is hereditarily Lindel\"of.
 The Filippov space is a split square $[0,1]^2$ where the splitting functions are 
$f_\xi: [0,1]^2\setminus\{r_\xi\}\rightarrow \T$   given by $f_\xi(x)={{x-r_\xi}\over {||x-r_\xi||}}$
where, $\T$ is the unit sphere in $\R^2$.  By Proposition \ref{weaklyluzin} the hereditary Lindel\"of property
implies the required property of $\Phi_E$.
\end{proof}

Analyzing the proof of Theorem 3.3. of \cite{kunenloc} one sees that similar arguments
give e.g., nonmetrizable split intervals with the properties as in Theorem \ref{macon}
consistent with MA$+\neg$CH.

Recall that a subset $Y$ of a Banach space $X$ is called $r$-equilateral if and only if
$\|y_1-y_2\|=r$ for any two distinct $y_1, y_2\in Y$, it is equilateral if it is $r$-equilateral 
for some $r\in \R$. In \cite{equilateral} using unordered split intervals we consistently constructed 
examples of nonseparable $C(K)$s without uncountable equilateral sets. 
This implies that $K$ cannot have a compact subspace with a totally disconnected 
nonmetrizable continuous image $L$ because the functions $\chi_A-\chi_{L\setminus A}$
for clopen $A\subseteq L$
which form a $2$-equilateral set in $C(L)$ would give rise to 
an uncountable  $2$-equilateral set in $C(K)$.

However, we proved in \cite{equilateral}  that already
MA and the negation of CH implies that every nonseparable
 $C(K)$ contains in its unit sphere  an uncountable $2$-equilateral set. Theorem \ref{macon}
sheds more light on this topic. In \cite{equilateralcofk} S. Mercourakis and G. Vassiliadis
list (Theorem 2 and Corollary 1) many properties of a compact $K$ which imply
the existence of uncountable $2$-equilateral sets. Theorem \ref{macon} 
and Proposition \ref{properties} show that it is possible to have an uncountable 
$2$-equilateral set for none of these reasons:

\begin{corollary}\label{equilateral} It is consistent that there is a compact Hausdorff space $K$ 
such that $K$ is hereditarily separable, hereditarily Lindel\"of, does not carry
a Radon measure of uncountable type nor the $C(K)$ contains an isometric copy
of $C(L)$ for $L$ totally disconnected,
but $C(K)$ contains an uncountable $2$-equilateral set in the unit sphere.
\end{corollary}

\begin{theorem}\label{ocaresult} Assume OCA.  Suppose that $\kappa$,
$M, K_\xi$, $\{r_\xi:\xi<\kappa\}$ are as in Definition \ref{splitdefinition}.
Let $K$ be a nonmetrizable  split $M$ induced by splitting functions $f_\xi: M\setminus\{r_\xi\}\rightarrow K_\xi$.
Then $K$ has a compact subspace with nonmetrizable continuous totally disconnected image.
\end{theorem}
\begin{proof} As $K$ is nonmetrizable, uncountably many $K_\xi$s must be nondegenerate.
The first case is that there is an uncountable
$E\subseteq \kappa$ such that  $K_\xi$s contain just two points $t_\xi^1, t_\xi^2$ for all $\xi\in E$.
Then consider $\phi:K\rightarrow \{t_\xi^1, t_\xi^2\}^E$ defined by $\phi(x)(\xi)=x(\xi)$ for
all $x\in K$ and $\xi\in E$. Of course $\{t_\xi^1, t_\xi^2\}^E$ is totally disconnected and 
the image $\phi[K]$ is nonmetrizable since 
the points $\phi(x_{\xi, t_\xi^1})$, $\phi(x_{\xi, t_\xi^1})$ witness the fact that 
no countable set of coordinates can separate the points of $\phi[K]$.
In the second, the nontrivial case when uncountably many $M_\xi$s contain
more
than two points we  use Theorem 4.3. of \cite{kunenloc}. 
According to it when $\pi: K\rightarrow M$
is defined by $\pi(x)=x(\st)$ it is enough to find an uncountable $E\subseteq M$ 
and disjoint open $U_y^i$ for $y\in E$ and $i\in \{0,1,2\}$ with $\pi[U_y^i]\cap
\pi[Y_y^j]=\{y\}$ for any two distinct $i,j\in \{0,1,2\}$.  Let
$E$ be the set of all elements $\xi$ of $\kappa$ such that $M_\xi$ has at least three points.
Find open pairwise disjoint $U(\xi, i)\subseteq M_\xi$ for $i\in \{0,1, 2\}$ and all $\xi\in E$.
Now note that $U_{r_\xi}^i=(U(\xi, i))_{K,\xi}$ works. By 4.3 of \cite{kunenloc}
$K$ has an uncountable discrete set, so use Proposition \ref{properties} (1).
\end{proof}

By analyzing the proof of Theorem 4.3 of \cite{kunenloc} one notes that the applications of
OCA yield uncountable subsets  $E\subseteq \kappa$ and  $U_\xi\subseteq K_\xi$ for $\xi\in E$
such that $f_\xi(r_\eta)\not\in U_\xi$ for any distinct $\xi,\eta\in E$ (or
$f_\xi(r_\eta)\in U_\xi$ if $U_\xi$ is the complement of the previous $U_\xi$). 

On the other hand in \cite{equilateral}  we consistently constructed 
an unordered split interval $K$ induced by splitting functions $f_\xi: [0,1]\setminus\{r_\xi\}
\rightarrow [-1,1]$ for some
$r_\xi\in [0,1]$ and all $\xi<\omega_1$ such that given any $k\in \N$ and
any nonempty open subintervals $I_1, ..., I_k, J_1, ..., J_k$ of $[-1,1]$ and any collection
$\{\{\alpha_1^\xi, ..., \alpha_k^\xi\}: \xi<\omega_1\}$ of pairwise disjoint subsets 
of $\omega_1$,  there are $\xi<\eta<\omega_1$ such that for all $i\leq k$ we have
$$f_{\alpha^\xi_i}(r_{\alpha^\eta_i})\in I_i, \ {\rm and} \ 
f_{\alpha^\eta_i}(r_{\alpha^\xi_i})\in J_i.$$
This shows (see also 2.6 in \cite{kunenloc} and \S5 of \cite{abraham1}) 
that the notion of $(f_\xi)_{\xi<\kappa}$-entangled sets
make sense for arbitrary split compact space even in a nonsymmetric setting (i.e.,
when $I_i\not=J_i$.)

\bibliographystyle{amsplain}

\end{document}